\documentclass[oneside,draft,notref]{amsart}
\usepackage[T1]{fontenc}
\usepackage[latin9]{inputenc}
\usepackage{longtable}
\usepackage{float}
\usepackage{amsthm}
\usepackage{setspace}
\usepackage{amssymb}
\makeatletter

\newtheorem{observ}{Observation}[section]

\newtheorem{discussion}[observ]{Discussion}
\newtheorem{remark}{Remark}
\newtheorem{convention}{Convention}

\theoremstyle{plain}
\newtheorem{thm}{Theorem}
  \theoremstyle{definition}
  \newtheorem{defn}[thm]{Definition}
  \theoremstyle{remark}
  \newtheorem{claim}[thm]{Claim}
\newenvironment{lyxlist}[1]
{\begin{list}{}
{\settowidth{\labelwidth}{#1}
 \setlength{\leftmargin}{\labelwidth}
 \addtolength{\leftmargin}{\labelsep}
 }}
{\end{list}}
  \theoremstyle{plain}
  \newtheorem{lem}[thm]{Lemma}
  \theoremstyle{plain}
  \newtheorem{cor}[thm]{Corollary}
  \theoremstyle{plain}
  \newtheorem{fact}[thm]{Fact}

\usepackage{pb-diagram}

\AtBeginDocument{

}

\makeatother

\begin{document}

\author[M. Cohen]{Moran Cohen} \address{ Institute of Mathematics, The Hebrew University of Jerusalem, Jerusalem 91904, Israel. } \email{moranski@math.huji.ac.il}
\thanks{The authors thank Assaf Hasson for his constructive remarks}
\author[S. Shelah]{Saharon Shelah } \address{ Institute of Mathematics, The Hebrew University of Jerusalem,
 Jerusalem 91904, Israel and Department of Mathematics, Rutgers University, New Brunswick, NJ 08854, USA } \email{shelah@math.huji.ac.il} \thanks{This Research was funded partly by the ISF. Paper E65}

\title{Ranks for strongly dependent theories}
\begin{abstract}
There is much more known about the family of superstable theories
when compared to stable theories. This calls for a search of an analogous
{}``super-dependent'' characterization 
in the context of dependent
theories. This problem has been treated in \cite{Sh:783,Sh:863},
where the candidates 
\textquotedbl{}Strongly dependent\textquotedbl{},
\textquotedbl{}Strongly dependent$\,
^{2}$\textquotedbl{} and others
were considered. These families 
generated new families when we considering
intersections with the stable family. 
Here, continuing \cite[\S 2, \S 5E,F,G]{Sh:863},
 we deal with several 
candidates, defined using dividing 
properties and related ranks of
types. Those candidates are 
subfamilies of {}``Strongly dependent''.
fulfilling some promises from \cite{Sh:863}
in particular \cite[1.4(4)]{Sh:863},
we try to make this self contained
within reason
by repeating some things from there.
More specifically we fulfil some promises
from \cite{Sh:863} to
to give more details, in particular:
in \S4 for \cite[1.4(4)]{Sh:863},
in \S2  for \cite[5.47(2)=Ldw5.35(2)]{Sh:863}
and in \S1  for \cite[5.49(2)]{Sh:863}
\end{abstract}
\maketitle
\newpage{}

\section{Strongly dependent theories}

\discussion{The basic property from which this work is derived is strongly dependent$\,^{1}$,
it has been studied extensively in \cite{Sh:863}. For proofs and
more we refer to that article. We quote the necessary minimum in order
to build on that.}
\begin{defn}
\label{def:strong_dep_up}We say that $\kappa^{{\rm ict},1}(T):=\kappa^{{\rm ict}}(T)>\kappa$
if the set \[
\Gamma_{\overline{\varphi}}:=\left\{ \varphi_{i}(\overline{x}_{\eta},\overline{y}_{i}^{j})^{{\bf if}(\eta(i)=j)}:i<\kappa,\; j<\omega,\eta\in\ ^{\kappa}\omega\right\} \]
is consistent with $T$, for some sequence of formulas $\overline{\varphi}=\left\langle \varphi_{i}(\overline{x},\overline{y}_{i}):i<\kappa\right\rangle $.
We will say that $\kappa^{{\rm ict}}(T)=\kappa$ iff $\kappa^{{\rm ict}}(T)>\lambda$
holds for all $\lambda<\kappa$ but $\kappa^{{\rm ict}}(T)>\kappa$
does not.

$T$ is called strongly dependent$\ ^{1}$ if $\kappa^{{\rm ict}}(T)=\aleph_{0}$.
\end{defn}

\discussion{The following properties are used in connecting the new properties
with the original.}
\begin{claim}
\label{cla:cutting_indi_props(2)}~$T$ is not strongly dependent$\,^{1}$
iff there exist sequences \[\overline{\varphi}=\left\langle \varphi_{i}(\overline{x},\overline{y}_{i}):i<\kappa\right\rangle \; \; 
{\rm and} \; \;  \left\langle \overline{a}_{k}^{i}:i<\kappa,k<\omega\right\rangle \]
such that $\lg\overline{y}_{i}=\lg\overline{a}_{k}^{i}$, $\left\langle \overline{a}_{k}^{i}:k<\omega\right\rangle $
an indiscernible sequence over \[\cup\left\{ \overline{a}_{k}^{j}:j\neq i,\; j<\kappa,\; k<\omega\right\} \]
for all $i<\kappa$ it holds that $\left\{ \varphi_{i}(\overline{x},\overline{a}_{0}^{i})\wedge\neg\varphi_{i}(\overline{x},\overline{a}_{1}^{i}):i<\kappa\right\} $
is a type in $\mathfrak{C}$.\end{claim}
\begin{thm}
\label{thm:ind_spl_strong_indep_1}For a given (or any) $\alpha\geq\omega$
the following are equivalent\end{thm}
\begin{lyxlist}{00.00.0000}
\item [{$(1)$}] $T$ is strongly dependent$\ ^{1}$
\item [{$_{\alpha}(2)$}] For every $\overline{c}\subseteq\mathfrak{C}$
and indiscernible sequence $\left\langle \overline{a}_{t}:t\in I\right\rangle $
where $\lg(\overline{a}_{t})=\alpha$ the function $t\mapsto{\rm tp}(\overline{a}_{t},\overline{c})$
divides $I$ to finitely many convex components.
\item [{$_{\alpha}(2)^{\prime}$}] Same as $_{\alpha}(2)$ but with $\lg(\overline{c})=1$
\item [{$_{\alpha}(2)^{\prime\prime}$}] Same as $_{\alpha}(2)^{\prime}$
but with $I=\omega$.
\end{lyxlist}

\discussion{Now we turn to discuss the new properties: strongly dependent$\,_{\ell}$
and strongly dependent$\,_{\mathcal{A}}$.}

\subsection{The dividing properties}

\subsubsection*{Order-based indiscernible structures, forms and dividing}
\begin{convention}
We fix a set $\mathcal{A}\subseteq\mathcal{P}(\mathcal{M}_{\mu_{1},\mu_{2}}(\mu_{3}))$,
such that all $\mathtt{A}\in\mathcal{A}$ contains at least one $n$-ary
term, for $n>0$.\end{convention}
\begin{defn}
We call $\mathtt{A}\in\mathcal{A}$ a form, and we define
\end{defn}
\[
\mathtt{A}(I):=\left\{ \tau(\overline{t}):\;\overline{t}=\left\langle t_{i}:i<\mu\right\rangle \in{\rm incr}(I,\mu),\;\tau(\mu)\in\mathtt{A},\;\mu<\mu_{3}\right\} \]
for a linear order $I$.
\begin{defn}
We call $\overline{s}_{0},\overline{s}_{1}$ equivalent in $\mathtt{A}(I)$
iff there exist a term $\overline{\tau}\subseteq A$ and increasing
sequences $\overline{t}_{0},\overline{t}_{1}$ such that $\overline{s}_{i}=\overline{\tau}(\overline{t}_{i}),\;\left(i=0,1\right)$.

Let $E$ a convex equivalence relation on $I$ we say that $\overline{s}_{0},\overline{s}_{1}$
are equivalent in $\mathtt{A}(I,E)$ iff $\overline{s}_{0},\overline{s}_{1}$
are equivalent in $\mathtt{A}(I)$ and also $\overline{t}_{0},\overline{t}_{1}$
are equivalent relative to $E$.\end{defn}
\begin{convention}
We will limit the discussion to the case $\mathcal{A}\subseteq\mathcal{P}(\mathcal{M}_{\omega\omega}(\omega))$.
\end{convention}

\remark{Note that a form restricts both the terms which can be used as well
as the assignable tuples to those which preserve the same order structure.}

\discussion{We now turn to define the structure classes.}
\begin{defn}
\label{def:k^or}$\mathfrak{k}^{{\rm or}}$ Denotes the class of linear
orders with the dictionary $\left(I,<\right)$.

\label{def:k^or+or<n}$\mathfrak{k}^{{\rm or+or}(<n)}$ Denotes the
class of structures $\mathtt{M}(I)$ whose universe is the disjoint
union of a linear order $\left|I\right|$ with the set of increasing
sequences of length $<n$ in $I$ , and the dictionary is \[
\left(I\cup{\rm incr}(I,<n),<,S_{0}\ldots S_{n-1},R_{0}\ldots R_{n-1}\right)\]
where $<$ is binary, $S_{i}$ is unary, and $R_{i}$ binary such
that $\left(I,<\right)$ is a linear order. $S_{i}=\left\{ \overline{t}\in{\rm incr}(I,<n):\lg(\overline{t})=i\right\} $
for all $i<n$, $S_{i}(\overline{t})$ holds iff $\lg(\overline{t})=i$.
Also $R_{i}(\overline{t},t_{i})$ for all $i<\lg\left(\overline{t}\right)$
($t_{i}\in I,\:\overline{t}\in{\rm incr}(I,<n)$).

\end{defn}
\begin{convention}
In the above notation, $<n$ can be replaced with $\leq n$ to mean
$<n+1$.
\end{convention}

\discussion{We now turn to define the main properties with which we deal}
\begin{defn}
\label{def:ict_divide}We say that the type $p(\overline{x})$ does
${\rm ict}^{\ell}-\left(\Delta,n\right)$-divide over $A$ if\end{defn}
\begin{description}
\item [{For~$\ell=1$:}] There exist an indiscernible sequence\\ $\left\langle \overline{a}_{t}:t\in I\right\rangle =\overline{\mathbf{a}}\in{\rm Ind}_{\Delta}(\mathfrak{k}^{{\rm or}},A)$
and $s_{0}<_{I}t_{0}\leq_{I}s_{1}<_{I}t_{1}<_{I}\ldots s_{n-1}<_{I}t_{n-1}$
such that for any $\overline{c}$ which realizes $p$, ${\rm tp}_{\Delta}(\overline{c}\!^{\frown}\overline{a}_{s_{i}},A)\neq{\rm tp}_{\Delta}(\overline{c}\!^{\frown}\overline{a}_{t_{i}},A)$
holds for all $i<n$
\item [{For~$\ell=2$:}] There exist an indiscernible sequence\\ $\left\langle \overline{a}_{t}:t\in I\right\rangle =\overline{\mathbf{a}}\in{\rm Ind}_{\Delta}(\mathfrak{k}^{{\rm or}},A)$
and $s_{0}<_{I}t_{0}\leq_{I}s_{1}<_{I}t_{1}<_{I}\ldots s_{n-1}<_{I}t_{n-1}$
such that for any $\overline{c}$ which realizes $p$, \[
{\rm tp}_{\Delta}(\overline{c}\!^{\frown}\overline{a}_{s_{\ell}},A\cup\left\{ \overline{a}_{s_{j}}:j<\ell\right\} )\neq{\rm tp}_{\Delta}(\overline{c}\!^{\frown}\overline{a}_{t_{\ell}},A\cup\left\{ \overline{a}_{s_{j}}:j<\ell\right\} )\]
 holds for all $\ell<n$
\item [{For~$\ell=3$:}] There exist an indiscernible structure\\ $\left\langle \overline{a}_{t}:t\in I\cup{\rm incr}(<n,I)\right\rangle =\overline{\mathbf{a}}\in{\rm Ind}_{\Delta}(\mathfrak{k}^{{\rm or+or(<n)}},A)$
and $s_{0}<_{I}t_{0}\leq_{I}s_{1}<_{I}t_{1}<_{I}\ldots s_{n-1}<_{I}t_{n-1}$
such that for any $\overline{c}$ realizing $p$ and $\ell<n$: \[
{\rm tp}_{\Delta}(\overline{c}\!^{\frown}\overline{a}_{s_{\ell}},A\cup\overline{a}_{\left\langle s_{0}\ldots s_{\ell-1}\right\rangle })\neq{\rm tp}_{\Delta}(\overline{c}\!^{\frown}\overline{a}_{t_{\ell}},A\cup\overline{a}_{\left\langle s_{0}\ldots s_{\ell-1}\right\rangle })\]
holds.
\item [{For~$\ell=\mathcal{A}$:}] For some form $\mathtt{A}\in\mathcal{A}$
and indiscernible structure\\ $\overline{\mathbf{a}}=\left\langle \overline{a}_{t}:t\in\mathtt{A}(I)\right\rangle $
over $A$, $\left\langle \overline{a}_{t}:t\in\mathtt{A}(I,E)\right\rangle $
is not indiscernible over $A\cup\overline{c}$, for any $\overline{c}$
realizing $p$ and convex equivalence relation $E$ on $I$ with $\leq n$
equivalence classes.
\end{description}

\observ{$p(\overline{x})$ does ${\rm ict}^{4}-\left(\Delta,n\right)$-divide
over $A$ iff $p(\overline{x})$ does $\mathcal{A}-\left(\Delta,n\right)$-divide
over $A$ for $\mathcal{A}=\left\{ \mathtt{A}_{n}=\left\{ f_{i}(0,\ldots,i-1):1<i<n\right\} :n<\omega\right\} $.}

\observ{If $\mathcal{A}\subseteq\mathcal{A}^{\prime}$ and $p(\overline{x})$
does ${\rm ict}^{\mathcal{A}}-\left(\Delta,n\right)$-divide over
$A$, then $p(\overline{x})$ does ${\rm ict}^{\mathcal{A}^{\prime}}-\left(\Delta,n\right)$-divide
over $A$.}

\observ{If a type $p$ does ${\rm ict}^{1}-n(\ast)$-divide over $A$ then
$p$ does ${\rm ict}^{\mathcal{A}}-n(\ast)$-divide over $A$.}

\observ{If the type $p$ does ${\rm ict}^{\ell}-n(\ast)$-divide over $A$
then $p$ does ${\rm ict}^{\ell+1}-n(\ast)$-divide over $A$ ($1\leq\ell\leq3$).}
\begin{defn}
We say that the type $p(\overline{x})$ does ${\rm ict}^{\ell}-\left(\Delta,n\right)$-fork
over $A$ if there exist formulas $\varphi_{i}(\overline{x},\overline{c}_{i}),\;(i<m)$
such that $p(\overline{x})\vdash\bigvee_{i<m}\varphi_{i}(\overline{x},\overline{c}_{i})$
and each $\varphi_{i}$ does ${\rm ict}^{\ell}-\left(\Delta,n\right)$
divide over $A$.
\end{defn}
~
\begin{defn}
\label{def:kappa_ict_l}We say that $\kappa_{\mathrm{ict},\ell}(T)>\kappa$
if some type $p$ of $T$ does ${\rm ict}^{\ell}-n$-fork over $A$,
for all $n<\omega$ and $A\subseteq{\rm Dom}(p)$ of power $<\kappa$.
\end{defn}
~
\begin{defn}
\label{def:strong_dep_dn}We call $T$ strongly dependent$\,_{\ell}$
($\,_{\mathcal{A}}$) iff $\kappa_{{\rm ict},\ell}(T)=\aleph_{0}$
($\kappa_{{\rm ict},\mathcal{A}}(T)=\aleph_{0}$)
\end{defn}

\observ{If $p(\overline{x})$ does ${\rm ict}^{\ell}-\left(\Delta,n\right)$-fork
over $A$ then $p(\overline{x})$ does ${\rm ict}^{\ell}-\left(\Delta,k\right)$
fork over $A$ for all $k<n$.}

\observ{(finite character) if the type $p(\overline{x})$ does ${\rm ict}^{\ell}-\left(\Delta,n(\ast)\right)$-divide
over $A$ then $q$ does ${\rm ict}^{\ell}-\left(\Delta,n\right)$-divide
over $A$ for some finite $q\subseteq p$.}
\begin{claim}
\label{cla:canonical_witness_dividing} If $p(\overline{x})$ does
${\rm ict}^{\ell}-\left(\Delta,n(\ast)\right)$-divide over $A$,
it is possible to find witnesses as follows:\end{claim}
\begin{description}
\item [{Case~$\ell=1$:}] There exist $\overline{\mathbf{a}}=\left\langle \overline{a}_{n}:n<\omega\right\rangle \in{\rm Ind}(\mathfrak{k}^{{\rm or}},A)$,
$\overline{s}$ a sequence of length $n(\ast)$ from $\omega$ such
that $s_{0}=0,\;1\leq s_{n+1}-s_{n}\leq2$ and formulas $\left\langle \varphi_{i}(\overline{y},\overline{x},\overline{c}):i<i(\ast)\right\rangle $
$\overline{c}\in A$ such that\[
p(\overline{x})\vdash\bigvee_{i<i(\ast)}\left(\varphi_{i}(\overline{a}_{s_{n}},\overline{x},\overline{c})\wedge\neg\varphi_{i}(\overline{a}_{s_{n}+1},\overline{x},\overline{c})\right)\]
 for all $n<n(\ast)$.
\item [{Case~$\ell=2$:}] There exist $\overline{\mathbf{a}}=\left\langle \overline{a}_{n}:n<\omega\right\rangle \in{\rm Ind}(\mathfrak{k}^{{\rm or}},A)$,
$\overline{s}$ as in $\ell=1$ and formulas $\left\langle \varphi_{i}^{n}(\overline{y}_{0}\ldots\overline{y}_{n-1},\overline{x},\overline{c}):i<i(\ast),n<n(\ast)\right\rangle $
$\overline{c}\in A$ such that \[
p(\overline{x})\vdash\bigvee_{i<i(\ast)}\left(\varphi_{i}^{n}(\overline{a}_{s_{0}}\ldots\overline{a}_{s_{n-1}}\overline{a}_{s_{n}},\overline{x},\overline{c})\wedge\neg\varphi_{i}(\overline{a}_{s_{0}}\ldots\overline{a}_{s_{n-1}}\overline{a}_{s_{n}+1},\overline{x},\overline{c})\right)\]
 for all $n<n(\ast)$.
\item [{Case~$\ell=3$:}] There exist \[\overline{\mathbf{a}}=\left\langle \overline{a}_{t}:t\in\omega\cup{\rm incr}(<n(\ast),\omega)\right\rangle \in{\rm Ind}(\mathfrak{k}^{{\rm or+or(<n(\ast))}},A),\]
$\overline{s}$ as in $\ell=1$ and formulas $\left\langle \psi_{i}^{n}(\overline{y},\overline{z},\overline{x},\overline{c}):i<i(\ast),n<n(\ast)\right\rangle $
such that \[
p(\overline{x})\vdash\bigvee_{i<i(\ast)}\left(\psi_{i}^{n}(\overline{a}_{\left\langle s_{0}\ldots s_{n-1}\right\rangle },\overline{a}_{s_{n}},\overline{x},\overline{c})\wedge\neg\psi_{i}^{n}(\overline{a}_{\left\langle s_{0}\ldots s_{n-1}\right\rangle },\overline{a}_{s_{n}+1},\overline{x},\overline{c})\right)\]
for all $n<n(\ast)$.
\item [{Case~$\mathcal{A}$:}] There exist $\mathtt{A}\in\mathcal{A}$,
$m_{\ast}<\omega$, $\overline{\mathbf{a}}=\left\langle \overline{a}_{t}:t\in\mathtt{A}(\omega)\right\rangle $
indiscernible over $A$, sequences $\left\langle \overline{s}_{0,E},\overline{s}_{1,E}\in\mathtt{A}(m_{\ast},E):E\in{\rm ConvEquiv}(m_{\ast},n(\ast))\right\rangle $,
$\overline{b}\in A$ and formulas $\left\langle \psi_{E,i}(\overline{x},\overline{y}_{E,i},\overline{b}):E\in{\rm ConvEquiv}(m_{\ast},n(\ast)),\: i<i_{E}\right\rangle $
such that \[
\overline{p}(\overline{x})\vdash\bigvee_{i<i_{E}}\psi_{E,i}(\overline{x},\overline{a}_{\overline{s}_{0,E}},\overline{b})\equiv\neg\psi_{E,i}(\overline{x},\overline{a}_{\overline{s}_{1,E}},\overline{b})\]
holds for all $E\in{\rm ConvEquiv}(m_{\ast},n(\ast))$.\end{description}
\begin{proof}
~
\begin{description}
\item [{For~$\ell=1,2,3$:}] Easy, so we only give a summary. By \ref{cor:extend_indisc_seq}
it follows that there exists a dense extension $I^{\prime}$ of $I$
without endpoints such that $\left\langle \overline{a}_{t}:t\in I^{\prime}\right\rangle $
is an indiscernible  structure (for the corresponding $\ell$) over
$A$. Let $s_{0}<t_{0}\leq\ldots\leq s_{n-1}<t_{n-1}$ from $I$ witness
the dividing as in the definition. These indices can also be used
to show that $I^{\prime}$ is a witness of dividing. Similarly we
can choose an increasing $\left\langle r_{n}:n<\omega\right\rangle $
from $I^{\prime}$ such that $\left\{ s_{i},t_{i}:i<n-1\right\} \triangleleft\left\langle r_{n}:n<\omega\right\rangle \subseteq I$,
to get a witness based on $\omega$.
\item [{For~$\mathcal{A}$:}] Assume towards contradiction that the claim
does not hold. So we can choose\end{description}
\begin{enumerate}
\item A type $p$ which does $\left(\Delta,n(\ast)\right)$-fork over $A$ 
\item A linear order $I$.
\item An indiscernible structure $\left\langle \overline{a}_{t}:t\in\mathtt{A}(I)\right\rangle $
over $A$ witnessing 1.
\item $\overline{c}$ realizing $p$ 
\end{enumerate}
Such that for every finite $S\subseteq I$ there exists a convex equivalence
relation $E_{S}$ on $I$ with $\leq n(\ast)$ equivalence classes
such that ${\rm tp}_{\Delta}(\overline{a}_{\overline{s}_{0}},A\cup\overline{c})={\rm tp}_{\Delta}(\overline{a}_{\overline{s}_{1}},A\cup\overline{c})$
holds for any equivalent $\overline{s}_{0},\overline{s}_{1}\in\mathtt{A}(S,E_{S})$

Now let $\mathcal{D}$ an ultrafilter on 
$\left[I\right]^{<\omega}$
extending $\left\{ G_{S}:S\in\left[I\right]^{<\omega}\right\}$,
  where \[ G_{S}:=\left\{ T\in\left[I\right]^{<\omega}:S\subseteq T\right\} \in\mathcal{D}.\]
For all $S$ define the 2-sort model (with the sorts $M,I$) \[M_{S}:=\left\langle M,I,E,\left\langle f_{\tau,i}:\tau(\overline{x}_{\tau})\in\mathtt{A},\; i<n_{\tau}\right\rangle ,\overline{c}\right\rangle \]
where
\begin{enumerate}
\item $M,I$ as defined
\item $E^{M_{S}}=E_{S}$ an equivalence relation.
\item Since $\left\langle \overline{a}_{t}:t\in\mathtt{A}(I)\right\rangle $
is indiscernible, for every term $\tau(\overline{x}_{\tau})\in\mathtt{A}(\overline{x})$
we can define $n_{\tau}<\omega$ such that $n_{\tau}=\lg\left(\overline{a}_{\tau(\overline{u})}\right)$
for all $\overline{u}\in^{\lg(\overline{x}_{\tau})}\left[I\right]$
. We define for each term $\tau(\overline{x}_{\tau})\in\mathtt{A}(\overline{x})$
and $i<n_{\tau}$:

\begin{eqnarray*}
f_{\tau(\overline{x}_{\tau}),i}:\,^{\lg(\overline{x}_{\tau})}\left[I\right] & \to & M\\
\tau(\overline{u}) & \mapsto & \left(a_{\tau(\overline{u})}\right)_{i}\end{eqnarray*}

\end{enumerate}
Now, consider $N=\left(\prod_{S\in\left[I\right]^{<\omega}}M_{S}\right)/\mathcal{D}$.
\relax From the properties of ultraproducts it is easy to show that the functions
\begin{eqnarray*}
h:M\oplus I & \to & N\\
a & \mapsto & \left\langle a\right\rangle _{S\in\left[I\right]^{<\omega}}/\mathcal{D}\end{eqnarray*}
 fulfill
\begin{enumerate}
\item $h\upharpoonright\left\langle M,\overline{c}\right\rangle :\left\langle M,\overline{c}\right\rangle \to N\upharpoonright\mathcal{L}_{T}\cup\left\{ \overline{c}\right\} $
is elementary.
\item $h(f_{\tau,i}^{M_{S}}(\overline{u}))=f_{\tau}^{N}(h(\overline{u}))$.
\item $E^{N}\circ h$ is a convex equivalence relation on $I^{N}$ with
$\leq n(\ast)$ classes.
\item ${\rm tp}_{\Delta}\left(\overline{a}_{\overline{s}_{0}},A\cup\overline{c},M\right)={\rm tp}_{\Delta}\left(\overline{a}_{\overline{s}_{1}},A\cup\overline{c},M\right)$
holds for every pair of equivalent $\overline{s}_{0},\overline{s}_{1}\in\mathtt{A}(I,E^{N}\circ h)$.
\end{enumerate}
Contradicting that $\left\langle \overline{a}_{t}:t\in\mathtt{A}(I)\right\rangle $
witnesses that $p$ does $\left(\Delta,n(\ast)\right)$-divide over
$A$.

Now we show that it is possible to choose $I=\omega$. From \ref{cor:extend_indisc_seq}
there exists an extension $J$ of $I$ without endpoints, such that
$\left\langle \overline{a}_{t}:t\in\mathtt{A}(J)\right\rangle $ is
indiscernible, extending $\left\langle \overline{a}_{t}:t\in\mathtt{A}(I)\right\rangle $.
Let $\left\langle s_{i}:i<\omega\right\rangle $ increasing in $J$
such that $\left\langle s_{0}\ldots s_{\left|S\right|}\right\rangle $
enumerates $S$ above. We define $\overline{b}_{\tau(\overline{u})}=\overline{a}_{\tau(\overline{s}_{\overline{u}})}$
for all $\overline{u},\tau\in\mathtt{A}$. by the conclusion of the
claim it is easy to verify that $\left\langle \overline{b}_{t}:t\in\mathtt{A}\left(\omega\right)\right\rangle $
is a witness as required. 

Now, since for any $\overline{s}_{0},\overline{s}_{1}\in\mathtt{A}(S)$
it holds that $\overline{s}_{0},\overline{s}_{1}$ are equivalent
in $\mathtt{A}(I,E)$ iff they are equivalent in $\mathtt{A}(S,E\upharpoonright S)$,
so for some $m_{\ast}<\omega$ such that $S\subseteq m_{\ast}$ we
can choose two equivalent (in $\mathtt{A}(\omega,E)$) $\overline{s}_{0},\overline{s}_{1}\in\mathtt{A}(m_{\ast})$
with $\overline{b}_{\overline{s}_{0}},\:\overline{b}_{\overline{s}_{1}}$
having different types over $A$ based only on $E\upharpoonright m_{\ast}$.

\end{proof}
We use the following freely

\observ{If $p(\overline{x})$ does ${\rm ict}^{\ell}-n$ divide over $A$
then $p(\overline{x})$ does ${\rm ict}^{\ell}-n$-divide over $B$
for every $B\subseteq A$ .}

\subsection{Strongly dependent$\,_{1}$ $\Rightarrow$ Strongly dependent$\ ^{1}$}

\discussion{Claim \ref{cla:sdep_down_to_sdep_up} is a connection to \cite{Sh:863}.}
\begin{claim}
\label{cla:sdep_down_to_sdep_up}$T$ is strongly dependent$\!_{1}$
(Definition \ref{def:strong_dep_dn}) $\Rightarrow$ $T$ is strongly
dependent$\ ^{1}$ (Definition \ref{def:strong_dep_up})\end{claim}
\begin{defn}
For a set of formulas $\mathcal{Q}$, define the formula \[
{\rm Even}\mathcal{Q}:=\bigvee\left\{ \bigwedge_{q\in\mathcal{Q}}q^{{\rm if}(q\in u)}:u\in\left[\mathcal{Q}\right]^{r},\;2|r,\; r\leq\left|\mathcal{Q}\right|\right\} \]

\end{defn}

\remark{$Even\mathcal{Q}$ is true iff the number of true sentences in $\mathcal{Q}$
is even.}
\begin{proof}
Assume that $T$ is not strongly dependent$\,^{1}$: by $_{\alpha}(2)^{\prime\prime}$
of theorem \ref{thm:ind_spl_strong_indep_1} there exist an indiscernible
sequence $\left\langle \overline{a}_{n}:n<\omega\right\rangle $ $(\lg\overline{a}_{n}=\omega)$
and an element $c$ such that ${\rm tp}(\overline{a}_{n},c)\neq{\rm tp}(\overline{a}_{n+1},c)$
for all $n<\omega$. consider $p(x):={\rm tp}(c,\cup\left\{ \overline{a}_{n}:n<\omega\right\} )$.
Fix a finite $A\subseteq{\rm Dom}(p)$. We need to show that $p$
does ${\rm ict}^{1}-n(\ast)$-fork over $A$ for some $n(\ast)$,
however we can prove this for any $1<n(\ast)<\omega$. Fix $n(\ast)$
and let $\overline{u}\subseteq I$ increasing and finite such that
$A\subseteq\cup\left\{ \overline{a}_{u_{i}}:i<\lg\overline{u}\right\} $.
Let $m=\max\overline{u}+1$. So $\left\langle \overline{a}_{n}:m\leq n<\omega\right\rangle $
is indiscernible over $A$. since for all $n\geq m$ there exists
$\varphi_{n}(\overline{x},y)$ such that $\models\varphi_{n}(\overline{a}_{n},c)\wedge\neg\varphi_{n}(\overline{a}_{n+1},c)$,
we get that $\varphi_{n}(\overline{a}_{n},x)\wedge\neg\varphi_{n}(\overline{a}_{n+1},x)\in p(x)$. 

Define a map $f:\left[\omega\right]^{2}\to\left\{ {\bf t},{\bf f}\right\} ^{4}$
as follows $f(\left\{ i,j\right\} )=(s_{0},s_{1},s_{2},s_{3})$ where
w.l.o.g $i<j$ and $s_{k}(k<4)$ are truth values such that\[
\models\varphi_{m+2i}(\overline{a}_{m+2j})^{s_{0}}\wedge\varphi_{m+2i}(\overline{a}_{m+2j+1})^{s_{1}}\wedge\varphi_{m+2j}(\overline{a}_{m+2i})^{s_{2}}\wedge\varphi_{m+2j}(\overline{a}_{m+2i+1})^{s_{3}}\]

By Ramsey's theorem, there exists an infinite $S\subseteq\omega$
such that $f\upharpoonright\left[S\right]^{2}$ is constant with value
$(s_{0},s_{1},s_{2},s_{3})$. Let $\left\langle i_{n}:n<n(\ast)\right\rangle $
enumerate $S$ in increasing order.

Define $\psi(x,\overline{y})$ as follows:
\begin{lyxlist}{00.00.0000}
\item [{if~$s_{0}=s_{1}\wedge s_{2}=s_{3}$}] let $\psi(x,\overline{y}):=Even\left\{ \varphi_{m+2i_{n}}(\overline{y},x):n<n(\ast)\right\} $. 
\item [{if~$s_{0}\neq s_{1}$}] let $\psi(x,\overline{y}):=\varphi_{m}(\overline{y},x)$. 
\item [{if~$s_{0}=s_{1}\wedge s_{2}\not=s_{3}$}] let $\psi(x,\overline{y}):=\varphi_{m+2i_{n-1}}(\overline{y},x)$.
\end{lyxlist}
Now let $\vartheta(x):=\bigwedge_{n<n(\ast)}\psi(x,\overline{a}_{m+2i_{n}})\Delta\psi(x,\overline{a}_{m+2i_{n}+1})$.
It is easy to verify that $\models\psi(c,\overline{a}_{m+2i_{n})})\equiv\neg\psi(c,\overline{a}_{m+2i_{n}+1})$
holds for any $n<n(\ast)$, so

$p\vdash\vartheta$. Now $\vartheta$ does ${\rm ict}^{1}-(\psi,n(\ast))$-divide
over $A$:

Choose a finite $u\subseteq\lg\overline{a}$ and let $\psi^{\prime}(x,\overline{y}\upharpoonright u)=\psi(x,\overline{y})$.
So $\vartheta(x)\vdash\psi^{\prime}(x,\overline{a}_{m+2i_{n}}\upharpoonright u)\equiv\neg\psi^{\prime}(x,\overline{a}_{m+2i_{n}+1}\upharpoonright u)$
holds for the indiscernible sequence $\left\langle \overline{a}_{n}\upharpoonright u:m\leq n<\omega\right\rangle $
and elements $s_{n}=m+2i_{n},t_{n}=m+2i_{n}+1$.
\end{proof}

\section{\label{sub:ranks}Ranks}
\begin{defn}
We define the ranks ${\rm ict}^{\ell}-{\rm rk}_{P}^{m}\;(P\in\left\{ fork,div\right\} )$
on the class of $m$-types of $T$ ( $m<\omega$ ) as follows:\end{defn}
\begin{itemize}
\item ${\rm ict}^{\ell}-{\rm rk}_{P}^{m}(p(\overline{x}))\geq0$ for all
$m$-types.
\item For a given ordinal $\alpha$, ${\rm ict}^{\ell}-{\rm rk}_{P}^{m}(p(\overline{x}))\geq\alpha$
if for all $q\subseteq p$, $A\subseteq{\rm Dom}(p)$ and $n<\omega$
($q,A$ finite) and $\beta<\alpha$, for some extension $q^{\prime}\supseteq q$
it holds that ${\rm ict}^{\ell}-{\rm rk}_{P}^{m}(q^{\prime})\geq\beta$
and also:

\begin{lyxlist}{00.00.0000}
\item [{For~$P=fork$:}] $q^{\prime}$ does ${\rm ict}^{\ell}-\left(\mathcal{L},n\right)$-fork
over $A$.
\item [{For~$P=div$:}] $q^{\prime}$ does ${\rm ict}^{\ell}-\left(\mathcal{L},n\right)$-divide
over $A$.
\end{lyxlist}
\item If $P=fork$ we omit $P$.
\end{itemize}

\observ{${\rm ict}^{\ell}-{\rm rk}^{m}(p)\geq{\rm ict}^{\ell}-{\rm rk}_{div}^{m}(p)$
for any $m$-type $p$.}

\observ{\label{obs:complete_ext_same_rank}For an $m$-type $p$ over $B$
such that ${\rm ict}^{\ell}-{\rm rk}^{m}(p)=\alpha$ there exists
an extension $p\subseteq q\in\mathbf{S}^{m}(B)$, a complete type
of the same rank .}
\begin{proof}
Identical to \cite[Theorem II.1.6, p.24]{Sh:c}.\end{proof}
\begin{convention}
We denote for the rest of this section\begin{eqnarray*}
\lambda_{\ell} & = & \left|T\right|\\
\lambda_{\mathcal{A}} & = & \left|T\right|+\sum_{\mathtt{A}\in\mathcal{A}}\aleph_{0}^{\left|\mathtt{A}\right|}\end{eqnarray*}
\end{convention}
\begin{lem}
\label{lem:infi_rk_type}If ${\rm ict}^{\ell}-{\rm rk}^{m}(\overline{x}=\overline{x})\geq\lambda_{\ell}^{+}$
then there exists $p\in\mathbf{S}^{m}(A)$ which does ${\rm ict}^{\ell}-n(\ast)$-divide
over $B$ for all $n(\ast)<\omega,\; B\in\left[A\right]^{<\omega}$.\end{lem}
\begin{proof}
We prove for $\ell=1$ and $\ell=\mathcal{A}$ (the cases $\ell=2,3$
are analogous to $\ell=1$). 

We choose, for each $\eta\in{\rm ds}(\lambda_{\ell}^{+})$, by induction
on $\lg(\eta)$ the following objects:
\begin{description}
\item [{Case~$\ell=1$:}]~

\begin{multline*}
p_{\eta},\; k_{\eta},\;\overline{b}_{\eta},\:\overline{c}_{\eta}\\
\left\langle \varphi_{\eta,k}(\overline{x},\overline{y}_{\eta}),\overline{\mathbf{a}}_{\eta,k}=\left\langle \overline{a}_{\eta,k,t}:t\in\omega\right\rangle ,\;\overline{s}_{\eta,k}:k<k_{\eta}\right\rangle \\
\left\langle \overline{\psi}_{\eta,k,i}(\overline{z}_{\eta,k,i},\overline{y}_{\eta},\overline{x}):k<k_{\eta},i<\lg\left(\overline{s}_{\eta,k}\right)\right\rangle \end{multline*}

\item [{Case~$\ell=\mathcal{A}$:}]~

\begin{multline*}
p_{\eta},\; k_{\eta},\;\overline{b}_{\eta},\:\overline{c}_{\eta}\\
\left\langle \varphi_{\eta,k}(\overline{x},\overline{y}_{\eta}),\overline{\mathbf{a}}_{\eta,k}=\left\langle \overline{a}_{\eta,k,t}:t\in\mathtt{A}_{\eta,k}(\omega)\right\rangle ,\; m_{\eta,k}:k<k_{\eta}\right\rangle \\
\biggl<\overline{s}_{E,0}^{\eta,k},\;\overline{s}_{E,1}^{\eta,k},\;\overline{\psi}_{\eta,k,E}(\overline{z}_{\eta,k,E},\overline{y}_{\eta},\overline{x}):\\
k<k_{\eta},\; E\in{\rm ConvEquiv}(m_{\eta,k},\lg(\eta))\biggr>\end{multline*}
such that
\begin{itemize}
\item $p_{\left\langle \ \right\rangle }=\emptyset,\,\overline{b}_{\left\langle \ \right\rangle }=\left\langle \ \right\rangle ,\, k_{\left\langle \ \right\rangle }=0$.
\item $\overline{c}_{\eta}$ realizes $p_{\eta}$.
\item $p_{\eta}$ is a finite type, ${\rm ict}^{\ell}-{\rm rk}^{m}(p_{\eta})\geq\min\left({\rm Rang}(\eta)\cup\left\{ \lambda_{\ell}^{+}\right\} \right)$
for all $\eta\in{\rm ds}(\lambda_{\ell}^{+})$. 
\item $p_{\eta}\vdash\bigvee_{k<k_{\eta}}\varphi_{\eta,k}(\overline{x},\overline{b}_{\eta})$.
\item For $\eta=\nu\!^{\frown}\left\langle \alpha\right\rangle $:

\begin{itemize}
\item $p_{\eta\!^{\frown}\left\langle \alpha\right\rangle }\supseteq p_{\eta}$
\item $\overline{b}_{\nu}\prec\overline{b}_{\eta}$.
\item $p_{\eta}$ does ${\rm ict}^{\ell}-\lg(\eta)$-fork over $\overline{b}_{\nu}$.
In particular $\varphi_{\eta,k}(\overline{x},\overline{b}_{\eta})$
does ${\rm ict}^{\ell}-\lg(\eta)$-divide over $\overline{b}_{\nu}$
for $k<k_{\eta}$. Moreover,

\begin{itemize}
\item Case $\ell=1$: $\overline{\psi}_{\eta,k,i}$ is a finite sequence
of formulas, and\[
\varphi_{\eta,k}(\overline{x},\overline{b}_{\eta})\vdash\bigvee_{\psi\in\overline{\psi}_{\eta,k,i}}\left[\psi(\overline{a}_{\eta,k,s_{i}},\overline{b}_{\eta},\overline{x})\equiv\neg\psi(\overline{a}_{\eta,k,s_{i}+1},\overline{b}_{\eta},\overline{x})\right]\]
holds for $i<\lg(\eta)=\lg\left(\overline{s}_{\eta,k}\right)$.
\item Case $\ell=\mathcal{A}$: $\overline{\psi}_{\eta,k,E}$ is a finite
sequence of formulas, and \[
\varphi_{\eta,k}(\overline{x},\overline{b}_{\eta})\vdash\bigvee_{\psi\in\overline{\psi}_{\eta,k,E}}\left[\psi(\overline{a}_{\eta,k,\overline{s}_{0,E}^{\eta,k}},\overline{b}_{\eta},\overline{x})\equiv\neg\psi(\overline{a}_{\eta,k,\overline{s}_{1,E}^{\eta,k}},\overline{b}_{\eta},\overline{x})\right]\]
holds for every $E\in{\rm ConvEquiv}(m_{\eta,k},\lg(\eta)$ for some
equivalent sequences $\overline{s}_{0,E}^{\eta,k},\overline{s}_{1,E}^{\eta,k}\in\mathtt{A}_{\eta,k}(m_{\eta,k},E)$.
\end{itemize}
\end{itemize}
\end{itemize}
\end{description}

\subparagraph*{Choice of a tree of types with descending ranks}

For $\eta=\left\langle \ \right\rangle $ - clear. Now let $\eta\in{\rm ds}(\lambda_{\ell}^{+})$,
$\alpha<\min\left({\rm Rang}(\eta)\cup\left\{ \lambda_{\ell}^{+}\right\} \right)$,
and $p_{\eta}$ a finite rank such that ${\rm ict}^{\ell}-{\rm rk}^{m}(p_{\eta})\geq\min\left({\rm Rang}(\eta)\cup\left\{ \lambda_{\ell}^{+}\right\} \right)$.
By the definition of rank and since $p_{\eta},\;{\rm Dom}(p_{\eta})$
are finite, there exists $q\supseteq p_{\eta}$ which does ${\rm ict}^{\ell}-(\lg\eta+1)$-fork
over ${\rm Dom}(p_{\eta})$ with rank $\geq\alpha$. By the finite
character of forking, there exists a finite $p_{\eta\!^{\frown}\left\langle \alpha\right\rangle }\subseteq q$
which does ${\rm ict}^{\ell}-\lg\eta$-fork over $\overline{b}_{\eta}$,
extending $p_{\eta}$. On the other hand, \[
{\rm ict}^{\ell}-rk^{m}(p_{\eta\!^{\frown}\left\langle \alpha\right\rangle })\geq{\rm ict}^{\ell}-rk^{m}(q)\geq\alpha\]
holds, since $q\supseteq p_{\eta\!^{\frown}\left\langle \alpha\right\rangle }$.
By the definition of forking and \ref{cla:canonical_witness_dividing}
we get \[\left\langle \varphi_{\eta\!^{\frown}\left\langle \alpha\right\rangle ,k}(\overline{x},\overline{b}_{\eta\!^{\frown}\left\langle \alpha\right\rangle }):k<k_{\eta\!^{\frown}\left\langle \alpha\right\rangle }\right\rangle.\]
(We choose w.l.o.g $\overline{b}_{\eta\!^{\frown}\left\langle \alpha\right\rangle }\succ\overline{b}_{\eta}$)
and the witnesses for ${\rm ict}^{\ell}-\lg(\eta)$-dividing of each
formula. This completes the iterated choice.

\subparagraph*{Choosing an infinite sequence}

We define for every $\eta\neq\left\langle \ \right\rangle $:

Case $\ell=1$:\[
\varrho_{\eta}:=\left(k_{\eta},\;\left\langle \varphi_{\eta,k}(\overline{x},\overline{y}_{\eta}),\; l_{\eta,k},\;\overline{s}_{\eta,k},\overline{\psi}_{\eta,k,i}(\overline{z}_{\eta,k,i},\overline{y}_{\eta},\overline{x}):k<k_{\eta}\right\rangle \right)\]
where $l_{\eta,k}=\lg\left(\overline{a}_{\eta,k,n}\right)$ for all
$n\in\omega$.

Case $\ell=\mathcal{A}$:\begin{multline*}
\varrho_{\eta}:=\biggl(k_{\eta},\;\left\langle \varphi_{\eta,k}(\overline{x},\overline{y}_{\eta}),\: l_{\eta,k}:\mathtt{A}_{\eta,k}\to\omega,\; m_{\eta,k}:k<k_{\eta}\right\rangle \,\\
\left\langle \overline{s}_{0,E}^{\eta,k},\overline{s}_{1,E}^{\eta,k},\overline{\psi}_{\eta,k,E}(\overline{z}_{\eta,k,E},\overline{y}_{\eta},\overline{x}):E\in\mathrm{ConvEquiv}(m_{\eta,k},\lg\left(\eta\right))\right\rangle \biggr)\end{multline*}
where $l_{\eta,k}$ is a function, mapping to each term $\tau(\overline{v})\in\mathtt{A}_{\eta,k}$
the length of $\overline{a}_{\eta,k,\tau(\overline{v})}$.

Now, there are at most $\lambda_{\ell}$ possibilities for the choice
of $\varrho_{\eta}$ since:

Case $\ell=1$: $k_{\eta},\: l_{\eta,k},\:\overline{s}_{\eta,k},\:\lg\left(\overline{y}_{\eta}\right),\:\lg\left(\overline{z}_{\eta,k,i}\right),\:\lg\left(\overline{\psi}_{\eta,k,i}\right)<\omega$
and so $\varrho_{\eta}$ has at most $\left|T\right|$ possibilities.

Case $\ell=\mathcal{A}$: $k_{\eta},\: m_{\eta,k}<\omega$. $l_{\eta,k}$
has at most $\sum_{\mathtt{A}\in\mathcal{A}}\aleph_{0}^{\left|\mathtt{A}\right|}$
possibilities and $\overline{s}_{0,E}^{\eta,k},\overline{s}_{1,E}^{\eta,k}$
have at most $\sum_{\mathtt{A}\in\mathcal{A}}\left|\mathtt{A}\right|$
possibilities. The formulas contain a finite number of variables,
so there are at most $\left|T\right|$ possibilities.

So by claim \ref{cla:find_in_ds} it follows that we can find a sequence
$\left\langle \varrho_{j}:j<\omega\right\rangle $ such that for any
$j_{\ast}<\omega$ there exists $\eta_{j_{\ast}}\in{\rm ds}(\lambda_{\ell}^{+})$
and $\varrho_{\eta_{j_{\ast}}\upharpoonright j}=\varrho_{j}$ holds
for all $j\leq j_{\ast}$. We denote the chosen objects as follows:

Case $\ell=1$:

\[
\varrho_{j}:=\left(k_{j},\;\left\langle \varphi_{j,k}(\overline{x},\overline{y}_{j}),\; l_{j,k},\;\overline{s}_{j,k},\overline{\psi}_{j,k,i}(\overline{z}_{j,k,i},\overline{y}_{j},\overline{x}):k<k_{j}\right\rangle \right)\]

Case $\ell=\mathcal{A}$:

\begin{multline*}
\varrho_{j}:=\biggl(k_{j},\;\left\langle \varphi_{j,k}(\overline{x},\overline{y}_{j}),\: l_{j,k}:\mathtt{A}_{j,k}\to\omega,\; m_{j,k}:k<k_{j}\right\rangle \,\\
\left\langle \overline{s}_{0,E}^{j,k},\;\overline{s}_{1,E}^{j,k},\;\overline{\psi}_{j,k,E}(\overline{z}_{j,k,E},\overline{y}_{j},\overline{x}):E\in\mathrm{ConvEquiv}(m_{j,k},j)\right\rangle \biggr)\end{multline*}

\subparagraph*{Using compactness to choose a new object}

We define a new dictionary $\tau_{\ast}$ by adding the constant symbols
to $\tau_{M}$: $\lg\overline{b}_{j}^{\ast}=\lg\overline{b}_{j},\:\lg\left(\overline{c}^{\ast}\right)=\lg\left(\overline{x}\right)$
and also

Case $\ell=1$: $\lg(\overline{a}_{j,k,t}^{\ast})=l_{j,k}$\[
\tau_{\ast}=\tau_{M}\cup\left\{ \overline{a}_{j,k,t}^{*}:t\in\omega,\, k<k_{j},\, j<\omega\right\} \cup\left\{ \overline{b}_{j}^{\ast}:j<\omega\right\} \cup\overline{c}^{\ast}\]

Case $\ell=\mathcal{A}$: $\lg(\overline{a}_{j,k,\tau(\overline{v})}^{\ast})=l_{j,k}(\tau(\overline{v}))$

\[
\tau_{\ast}=\tau_{M}\cup\left\{ \overline{a}_{j,k,t}^{*}:t\in\mathtt{A}_{j,k}(\omega),\, k<k_{j},\, j<\omega\right\} \cup\left\{ \overline{b}_{j}^{\ast}:j<\omega\right\} \cup\overline{c}^{\ast}\]

We now define families of formulas in $\mathcal{L}(\tau_{\ast})$,
for every $1\leq j<\omega$: \begin{eqnarray*}
\Delta_{j}^{{\rm type}} & = & \left\{ \bigvee_{k<k_{j}}\varphi_{j,k}(\overline{c}^{\ast},\overline{b}_{j}^{\ast})\right\} \end{eqnarray*}

Case $\ell=1$:\begin{multline*}
\Delta_{j}^{{\rm div}}:=\cup\left\{ {\rm Ind}(\overline{\mathbf{a}}_{j,k}^{\ast},\overline{b}_{j-1}^{\ast}):k<k_{j}\right\} \cup\biggl{\{}\left(\forall\overline{x}\right)\varphi_{j,k}(\overline{x},\overline{b}_{j}^{\ast})\to\\
\bigvee_{i<\lg(\overline{\psi}_{j,k,E})}\biggl(\psi_{j,k,i}(\overline{a}_{j,k,s_{j,k,i}}^{\ast},\overline{b}_{j-1}^{\ast},\overline{x})\equiv\neg\psi_{j,k,i}(\overline{a}_{j,k,s_{j,k,i}+1}^{\ast},\overline{b}_{j-1}^{\ast},\overline{x})\biggr):\\
E\in{\rm ConvEquiv}(m_{j,k},j),\: k<k_{j}\biggr\}\end{multline*}

Case $\ell=\mathcal{A}$:\begin{multline*}
\Delta_{j}^{{\rm div}}:=\cup\left\{ {\rm Ind}(\overline{\mathbf{a}}_{j,k}^{\ast},\overline{b}_{j-1}^{\ast}):k<k_{j}\right\} \cup\biggl{\{}\left(\forall\overline{x}\right)\varphi_{j,k}(\overline{x},\overline{b}_{j}^{\ast})\to\\
\bigvee_{i<\lg(\overline{\psi}_{j,k,E})}\biggl(\psi_{j,k,E,i}(\overline{a}_{j,k,\overline{s}_{0,E}^{j,k}}^{\ast},\overline{b}_{j-1}^{\ast},\overline{x})\equiv\neg\psi_{j,k,E,i}(\overline{a}_{j,k,\overline{s}_{1,E}^{j,k}}^{\ast},\overline{b}_{j-1}^{\ast},\overline{x})\biggr):\\
E\in{\rm ConvEquiv}(m_{j,k},j),\: k<k_{j}\biggr\}\end{multline*}

And define $\Delta_{j}=\Delta_{j}^{{\rm type}}\cup\Delta_{j}^{{\rm div}}$.
The collection $\Delta:=\bigcup_{j<\omega}\Delta_{j}$ is consistent
with $T$, since for all $j_{\ast}<\omega$, the assignment

\[
\overline{\mathbf{a}}_{\eta_{j_{\ast}}\upharpoonright j,k},\:\overline{b}_{\eta_{j_{\ast}}\upharpoonright j},\:\overline{c}_{\eta_{j_{\ast}}\upharpoonright j}\mapsto\overline{\mathbf{a}}_{j,k}^{\ast},\,\overline{b}_{j}^{\ast},\:\overline{c}^{\ast}\quad\left(j\leq j_{\ast}\right)\]

realizes $\bigcup_{j<j_{\ast}}\Delta_{j}$. 

\subparagraph*{Proving the chosen object is a counterexample, finishing the proof.}

Now, let\\ $\overline{\mathbf{a}}_{j,k}^{\ast},\,\overline{b}_{j}^{\ast}\subseteq\mathfrak{C}_{T}$
realizing $\Delta$ (recall that $\mathfrak{C}$ is sufficiently saturated)
and work again in $\tau_{T}$. To complete the proof we note the following:
\begin{itemize}
\item $p_{0}(\overline{x})=\left\{ \bigvee_{k<k_{j}}\varphi_{j,k}(\overline{x},\overline{b}_{j}^{\ast}):k<k_{j}\right\} $
is a type in $T$. 
\item The formula $\varphi_{j,k}(\overline{x},\overline{b}_{j}^{\ast})$
does ${\rm ict}^{\ell}-\left\langle \Delta,j\right\rangle $-divide
over $\overline{b}_{j-1}^{\ast}$ for all $k<k_{j},\;0<j<\omega$
.
\item For $\mathbf{S}^{m}(\bigcup_{j<\omega}\overline{b}_{j}^{\ast})\ni p\supseteq p_{0}$
, $n<\omega$ and finite $A\subseteq{\rm Dom}(p)$, there exists $n\leq j<\omega$
such that $A\subseteq\overline{b}_{j-1}^{\ast}$. Since $p$ is complete,
$p\vdash\bigvee_{k<k_{j}}\varphi_{j,k}(\overline{x},\overline{b}_{j}^{\ast})$
and ${\rm Dom}(p)$ contains the constants on the right hand, there
exists $k<k_{j}$ such that $p\vdash\varphi_{j,k}$. Since $\Delta_{j}^{{\rm div}}$
is realized, we get that $\varphi_{j,k}(\overline{x},\overline{b}_{j}^{\ast})$
does ${\rm ict}^{\ell}-j$-divide over $\overline{b}_{j-1}^{\ast}$,
and by monotonicity of dividing we get that $\varphi_{j,k}$ does
${\rm ict}^{\ell}-n$-divide over $A$. Therefore $p$ does also ${\rm ict}^{\ell}-n$
divide over $A$.
\end{itemize}
\end{proof}
\begin{cor}
\label{cor:inf_fork_to_inf_div}${\rm ict}^{\ell}-{\rm rk}^{m}(\overline{x}=\overline{x})\geq\infty\Rightarrow{\rm ict}^{\ell}-{\rm rk}_{div}^{m}(\overline{x}=\overline{x})\geq\infty$.\end{cor}
\begin{thm}
\label{thm:equivalence_rank_and_k_ict}For a first-order complete
$T$, TFAE:\end{thm}
\begin{enumerate}
\item $\kappa_{{\rm ict},\ell}(T)>\aleph_{0}$
\item ${\rm ict}^{\ell}-{\rm rk}^{m}(\overline{x}=\overline{x})=\infty$.
\item ${\rm ict}^{\ell}-{\rm rk}^{m}(\overline{x}=\overline{x})\geq\lambda_{\ell}^{+}$.
\item There exists a type $p(\overline{x})$ such that for all finite $A\subseteq{\rm Dom}(p)$,
$n_{\ast}<\omega$ it holds that $p$ does ${\rm ict}^{\ell}-n_{\ast}$
divide over $A$.\end{enumerate}
\begin{proof}
~
\begin{description}
\item [{$4\Rightarrow1$:}] Directly by the definitions.
\item [{$1\Rightarrow2$:}] For some type $p(\overline{x})$ for all finite
$A\subseteq\mathrm{Dom}(p)$, $n<\omega$ it holds that $p$ does
${\rm ict}^{\ell}-n$-fork over $A$. ${\rm ict}^{\ell}-{\rm rk}^{m}(p)\geq0$.
Assume that ${\rm ict}^{\ell}-{\rm rk}^{m}(p)\geq\alpha$ and we will
show that ${\rm ict}^{\ell}-{\rm rk}^{m}(p)\geq\alpha+1$. Let $q\subseteq p,\: A\subseteq\mathrm{Dom}(p),\; n<\omega$,
then $p$ extends $q$ and does ${\rm ict}^{\ell}-n$-fork over $A$.
Therefore ${\rm ict}^{\ell}-{\rm rk}^{m}(p)\geq\alpha+1$.
\item [{$2\Rightarrow3$:}] Clearly.
\item [{$3\Rightarrow4$:}] By Lemma \ref{lem:infi_rk_type}.
\end{description}
\end{proof}

\section{Equivalent definitions of {}``strongly dependent$\,_{\ell}(\,_{\mathcal{A}})$''
using automorphisms}

\discussion{It is useful to have an equivalent characterization of the strongly
dependent$\,_{\ell}(\,_{\mathcal{A}})$ properties using automorphisms.
This enables to work in a {}``pure model theoretic'' environment
when possible. What enables this equivalent characterization is a
sufficiently strongly saturated model where equivalence of types implies
existence of automorphisms of the model.}
\begin{defn}
The model $M$ is strongly $\kappa$-saturated if ${\rm tp}(\overline{a},M)={\rm tp}(\overline{b},M)$,
implies that $f(\overline{a})=\overline{b}$ for some $f\in{\rm Aut}(M)$,
for all $\overline{a},\overline{b}\in^{\gamma}\left|M\right|,\;\gamma<\kappa$.\end{defn}
\begin{claim}
\label{cla:strong_dn_autos}Let $M$ be strongly $\left(\kappa+\left|\mathcal{L}_{M}\right|\right)^{+}$-saturated.
Then ${\rm Th}(M)$ is strongly independent$\!^{1}$ iff for some
finite sequence $\overline{c}$ and $\left\langle \overline{a}_{\alpha,i}:i<\omega,\alpha<\kappa\right\rangle $
it holds that $\left\langle \overline{a}_{\alpha(\ast),i}:i<\omega\right\rangle $
is indiscernible over $\left\{ \overline{a}_{\alpha,i}:i<\omega,\alpha\neq\alpha(\ast)\right\} $
but $\pi(\overline{a}_{\alpha,0})\neq\overline{a}_{\alpha,1}$ for
all $\pi\in{\rm Aut}(M/\overline{c})$, $\alpha<\kappa$. \end{claim}
\begin{proof}
We use claim \ref{cla:cutting_indi_props(2)}. Indeed, assume that
${\rm Th}(M)$ is not strongly dependent$^{1}$. Therefore we can
find $\overline{\varphi}:=\left\langle \varphi_{i}(\overline{x},\overline{y}_{i}):i<\kappa\right\rangle $
such that the union of the set of formulas in the variables $\left\langle \overline{x}_{\alpha,i}:i<\omega,\alpha<\kappa\right\rangle $,
saying that $\left\langle \overline{x}_{\alpha(\ast),i}:i<\omega\right\rangle $
is an indiscernible sequence over \[\left\{ \overline{x}_{\alpha,i}:i<\omega,\alpha\neq\alpha(\ast)\right\} \; \; 
{\rm and\ } ]; \left\{ \varphi_{\alpha}(\overline{x},\overline{x}_{\alpha,0})\wedge\varphi_{\alpha}(\overline{x},\overline{x}_{\alpha,1}):\alpha<\kappa\right\} \]
is consistent. this is a family of formulas in $\kappa$ which is
realized in $M$, by saturation. Clearly no elementary map over $\overline{c}$
maps $\overline{a}_{\alpha,0}$ to $\overline{a}_{\alpha,1}$, for
any $\alpha<\kappa$. Conversely, if we can find $\left\langle \overline{a}_{\alpha,i}:i<\omega,\alpha<\kappa\right\rangle $
as above, it clearly follows by the strong saturation that ${\rm tp}(\overline{a}_{\alpha,0},\overline{c},M)\neq{\rm tp}(\overline{a}_{\alpha,1},\overline{c},M)$
for all $\alpha<\kappa$.
\end{proof}

\discussion{We now turn to strongly dependent$\,_{\ell},(\,_{\mathcal{A}})$.
By Theorem \ref{thm:equivalence_rank_and_k_ict}, being strongly independent$\,_{\ell}(\,_{\mathcal{A}})$
is equivalent to existence of $A,\overline{a}$ such that ${\rm tp}(\overline{a},B,\mathfrak{C})$
does ${\rm ict}^{\ell}-n$-divide over $B$ for any finite $B\subseteq A$
, $n<\omega$ . From this it follows that finding a characterization
by automorphisms for dividing is sufficient.}
\begin{claim}
\label{cla:dividing_dn_with_autos}Let $M$ be a strongly $\kappa$-saturated
model. For some $\overline{a},A\subset M,\:\left|\lg\overline{a}\right|+\left|A\right|<\kappa$
it holds that ${\rm tp}(\overline{a},A,M)$ does ${\rm ict}^{\ell}-n$-divide
(${\rm ict}^{\mathcal{A}}-n$-divide) strongly over $B$ if and only
if:\end{claim}
\begin{description}
\item [{Case~$\ell=1$:}] There exists an indiscernible sequence $\left\langle \overline{a}_{t}:t\in\omega\right\rangle $
over $B$ and a sequence $\overline{s}$ of length $n$ such that $1\leq s_{i+1}-s_{i}\leq2$
and for all $f\in{\rm Aut}(M/A)$, $g\in{\rm Aut}(M/B\cup f(\overline{a}))$
and $i<n$, it holds that $g(\overline{a}_{s_{i}})\neq\overline{a}_{s_{i}+1}$.
\item [{Case~$\ell=2$:}] There exists an indiscernible sequence $\left\langle \overline{a}_{t}:t\in\omega\right\rangle $
over $B$ and a sequence $\overline{s}$ of length $n$ such that
$1\leq s_{i+1}-s_{i}\leq2$ and for all $f\in{\rm Aut}(M/A)$, $i<n-1$
and $g\in{\rm Aut}(M/B\cup f(\overline{a})\cup\overline{a}_{s_{0}}\ldots\overline{a}_{s_{i-1}})$
it holds that $g(\overline{a}_{s_{i}})\neq\overline{a}_{s_{i}+1}$.
\item [{Case~$\ell=3$:}] There exists an indiscernible structure\\ $\left\langle \overline{a}_{t}:t\in\omega\cup{\rm incr}(<n,\omega)\right\rangle =\overline{\mathbf{a}}\in{\rm Ind}(\mathfrak{k}^{{\rm or+or(<n)}},A)$
and a sequence $\overline{s}$ of length $n$ such that $1\leq s_{i+1}-s_{i}\leq2$
and for all $f\in{\rm Aut}(M/A)$, $i<n-1$ and $g\in{\rm Aut}(M/B\cup f(\overline{a})\cup\overline{a}_{\left\langle s_{0}\ldots s_{i-1}\right\rangle })$
it holds that $g(\overline{a}_{s_{i}})\neq\overline{a}_{s_{i}+1}$.
\item [{Case~$\mathcal{A}$:}] There exist an indiscernible structure
$\left\langle \overline{a}_{t}:t\in\mathtt{A}(\omega)\right\rangle $
over $B$, $m<\omega$ and equivalent sequences $\overline{s}_{E,0},\overline{s}_{E,1}\in\mathtt{A}(\omega)$
for all $E\in{\rm ConvEquiv}(m,n)$ such that for all $f\in{\rm Aut}(M/A)$
and $g\in{\rm Aut}(M/B\cup f(\overline{a}))$ it holds that $g(\overline{a}_{\overline{s}_{E,0}})\neq\overline{a}_{\overline{s}_{E,1}}$.
\end{description}

\section{Preservation of strongly dependent under sums}
\begin{fact}
\label{fac:strongly_saturated_ultrapower}For a cardinal $\kappa$,
there exist a cardinal $\mu$ and ultrafilter $\mathcal{D}$ on $\mu$
such that for any model $M$, the ultrapower $M^{\mu}/\mathcal{D}$
is strongly $\kappa^{+}$-saturated.\end{fact}
\begin{defn}
Let $M,N$ be models in the same relational dictionary (i.e. no functions
or constants) $\tau$. We define new models $M\oplus N$ and $M+N$
as follows\end{defn}
\begin{itemize}
\item The universe of $M\oplus N$ is $\left|M\right|\cup\left|N\right|$
( w.l.o.g $\left|M\right|\cap\left|N\right|=\emptyset$). the dictionary
$\tau\cup\left\{ L,R\right\} $ where $L,R$ are unary relation, interpreting
$S^{M\oplus N}=S^{M}\cup S^{N}$ for every relation $S\in\tau$, and
$L^{M\oplus N}=\left|M\right|,R^{M\oplus N}=\left|N\right|$. 
\item $M+N=M\oplus N\upharpoonright\tau$\end{itemize}
\begin{claim}
\label{cla:ultrapow_functorial}For $\mathcal{D}$ an ultrafilter
on $I$ it holds that $\left(M\oplus N\right)^{I}/\mathcal{D}\simeq M^{I}/\mathcal{D}\oplus N^{I}/\mathcal{D}$\end{claim}
\begin{thm}
Let $M_{1},M_{2}$ models in a relational dictionary $\tau$. If ${\rm Th}(M_{1}),{\rm Th}(M_{2})$
are strongly dependent$^{1}$, then ${\rm Th}(M_{1}\oplus M_{2})$
is also strongly dependent$^{1}$.\end{thm}
\begin{proof}
By claim \ref{cla:ultrapow_functorial} and \ref{fac:strongly_saturated_ultrapower}
it follows that w.l.o.g $M_{1},M_{2},M_{1}\oplus M_{2}$ are strongly
$\kappa^{+}$-saturated. By claim \ref{cla:strong_dn_autos} there
exist $\left\langle \overline{a}_{\alpha,i}:\alpha<\kappa,j<\omega\right\rangle ,\overline{c}$
witnessing $\kappa^{{\rm ict}}({\rm Th}(M_{1}\oplus M_{2}))>\kappa$.
W.l.o.g $\overline{c}=\overline{c}^{1}\!^{\frown}\overline{c}^{2},\;\overline{a}_{\alpha,j}=\overline{a}_{\alpha,j}^{1}\!^{\frown}\overline{a}_{\alpha,j}^{2}$
such that $\overline{a}_{\alpha,j}^{i},\overline{c}^{i}\in M_{i}$.
Recall that $\left\langle \overline{a}_{\alpha(\ast),j}:j<\omega\right\rangle $
is an indiscernible sequence over $\left\{ \overline{a}_{\alpha j}:j<\omega,\alpha\neq\alpha(\ast)\right\} $
for $\alpha(\ast)<\kappa$, therefore $\left\langle \overline{a}_{\alpha(\ast),j}^{i}:j<\omega\right\rangle $
is indiscernible over $\left\{ \overline{a}_{\alpha j}^{i}:i<\omega,\alpha\neq\alpha(\ast)\right\} $.
Also, $f\in{\rm Aut}(M_{1}\oplus M_{2})$ iff there exist $f_{i}\in{\rm Aut}(M_{i})$
such that $f=f_{1}\cup f_{2}$ (as functions). Therefore, for some
$i=1,2$ and unbounded $S\subseteq\kappa$ it holds for all $\alpha\in S$
and for all $f_{i}\in{\rm Aut}(M_{i}/\overline{c}^{i})$ that $f_{i}(\overline{a}_{\alpha,0}^{i})\neq\overline{a}_{\alpha,1}^{i}$.
By Claim \ref{cla:strong_dn_autos} it follows that the sequences
$\left\{ \overline{a}_{\alpha j}:j<\omega,\alpha\in S\right\} $ are
witnesses for $\kappa^{{\rm ict},1}(M_{i})>{\rm otp}(S)=\kappa$.\end{proof}
\begin{thm}
(Case $\ell=1,2,3$) ${\rm T}h(M^{1}\oplus M^{2})$ is strongly dependent$\!_{\ell}$
iff\\ ${\rm Th}(M^{1}),\;{\rm Th}(M^{2})$ are strongly dependent$\!_{\ell}$.\end{thm}
\begin{proof}
{}``only if'' direction - assume w.l.o.g that ${\rm Th}(M^{1})$
is not strongly dependent$\!_{\ell}$. By lemma \ref{lem:infi_rk_type}
there exist $\overline{a}\in M^{1}$ and a set $A\subseteq M^{1}$
such that ${\rm tp}(\overline{a},A,M^{1})$ does ${\rm ict}^{\ell}-n$
divide over $B$ for any finite $B\subseteq A$ and $n<\omega$. This
easily implies that ${\rm tp}(\overline{a},A,M^{1}\oplus M^{2})$
does ${\rm ict}^{\ell}-n$ divide over $B$ for any finite $B\subseteq A$
and $n<\omega$, and so, ${\rm T}h(M^{1}\oplus M^{2})$ is not strongly
dependent$\!_{\ell}$. 

{}``if\textquotedbl{} direction - By \ref{lem:infi_rk_type}, there
exist $\overline{a}^{i}\in M^{i}$ and sets $A^{i}\subseteq M^{i}$
such that ${\rm tp}(\overline{a}^{1}\!^{\frown}\overline{a}^{2},A^{1}\cup A^{2},M^{1}\oplus M^{2})$
does ${\rm ict}^{\ell}-2\cdot n$ divide over $B^{1}\cup B^{2}$ for
all finite $B^{i}\subseteq A^{i}$ and $n<\omega$. If ${\rm tp}(\overline{a},A^{1},M^{1})$
does ${\rm ict}^{\ell}-n$ divide over $B^{1}$ for all $n<\omega$
and finite $B^{1}\subseteq A^{1}$ this concludes the proof. Otherwise,
there exist $n_{0}<\omega$ and finite $B^{1}\subseteq A^{1}$ such
that ${\rm tp}(\overline{a},A^{1},M^{1})$ does not ${\rm ict}^{\ell}-n_{0}$
divide over $B^{1}$. Since for all finite $B^{2}\subseteq A^{2},\: n>n_{0}$
it holds that ${\rm tp}(\overline{a}^{1}\!^{\frown}\overline{a}^{2},A^{1}\cup A^{2},M^{1}\oplus M^{2})$
does ${\rm ict}^{\ell}-2\cdot n$ divide over $B^{1}\cup B^{2}$,
we get by claim \ref{cla:div_sum_projects} that ${\rm tp}(\overline{a},A^{2},M^{2})$
does necessarily ${\rm ict}^{\ell}-n$ divide over $B^{2}$. Thus,
again by \ref{lem:infi_rk_type}, ${\rm Th}(M^{2})$ is not strongly
dependent$\,_{\ell}$.\end{proof}
\begin{fact}
\label{fac:theory_of_sum} $M\oplus N\equiv M^{\prime}\oplus N^{\prime}$
for models $M\equiv M^{\prime},N\equiv N^{\prime}$.\end{fact}
\begin{claim}
\label{cla:div_sum_projects} (Cases $\ell=1,2,3$) Let $\overline{a}^{i},A^{i},B^{i}\subseteq\left|M^{i}\right|,\;(i\in\left\{ 1,2\right\} )$,
then\\ ${\rm {\rm tp}}(\overline{a}^{1}\!^{\frown}\overline{a}^{2},A^{1}\cup A^{2},M^{1}\oplus M^{2})$
does ${\rm ict}^{\ell}-2n$-divide over $B^{1}\cup B^{2}$ iff ${\rm tp}(\overline{a}^{i},A^{i},M^{i})$
does ${\rm ict}^{\ell}-n$-divide over $B^{i}$, for some $i\in\left\{ 1,2\right\} $.\end{claim}
\begin{proof}
The proof for all the cases is analogous and the {}``if'' direction
is easy so we only give here the {}``only if'' of case $\ell=1$:
w.l.o.g $M^{1},M^{2},M^{1}\oplus M^{2}$ are strongly $\kappa^{+}$-saturated
and $\left|A^{1}\cup A^{2}\right|\leq\kappa$. By \ref{cla:dividing_dn_with_autos}
we can find $\left\langle \overline{a}_{t}^{1}\!^{\frown}\overline{a}_{t}^{2}:t\in\omega\right\rangle $,
an indiscernible sequence over $B^{1}\cup B^{2}$ and a sequence $\overline{s}$
of length $2n$ such that $1\leq s_{j+1}-s_{j}\leq2$ for all $j<2n$
and that $g(\overline{a}_{s_{j}}^{1}\!^{\frown}\overline{a}_{s_{j}}^{2})\neq\overline{a}_{s_{j}+1}^{1}\!^{\frown}\overline{a}_{s_{j}+1}^{2}$
holds for all $f\in{\rm Aut}\left(M^{1}\oplus M^{2}/A^{1}\cup A^{2}\right)$,
$g\in{\rm Aut}(M/B^{1}\cup B^{2}\cup f(\overline{a}^{1}\cup\overline{a}^{2}))$
and $j<2n$. 

Now, assume towards contradiction that $f_{i}\in{\rm Aut}(M^{i}/A^{i})\;(i=1,2)$
and that $g_{i}\in{\rm Aut}(M^{i}/B^{i}\cup f^{i}(\overline{a}^{i}))$
are such that $g_{i}(\overline{a}_{s_{j}}^{i})=\overline{a}_{s_{j}+1}^{i}$
holds for some $j<2n$. By the bijection $\Phi:{\rm Aut}(M^{1})\times{\rm Aut}(M^{2})\to{\rm Aut}(M^{1}\oplus M^{2})$,
we get that $f=f_{1}\cup f_{2}\in{\rm Aut}\left(M^{1}\oplus M^{2}/A^{1}\cup A^{2}\right)$
and that $g=g_{1}\cup g_{2}\in{\rm Aut}(M/B^{1}\cup B^{2}\cup f(\overline{a}^{1}\cup\overline{a}^{2}))$
- a contradiction. Thus, for all $j<2n$ there exists $i\in\left\{ 1,2\right\} $
such that $g(\overline{a}_{s_{j}}^{i})\neq\overline{a}_{s_{j}+1}^{i}$
holds for all $f\in{\rm Aut}(M^{i}/A^{i})$, $g\in{\rm Aut}(M^{i}/B^{i}\cup f^{i}(\overline{a}^{i}))$.
Denote by $i(j)$, the appropriate $i$ for every $j<2n$, . Let $i_{0}\in\left\{ 1,2\right\} $
be such that $S_{i_{0}}=\left\{ i(j)=i_{0}:j<2n\right\} $ has at
least $n$ elements. It now follows easily from \ref{cla:dividing_dn_with_autos}
that $\left\langle \overline{a}_{t}^{i_{0}}:t\in\omega\right\rangle $
are witnessing that ${\rm tp}(\overline{a}^{i_{0}},A^{i_{0}},M^{i_{0}})$
does ${\rm ict}^{\ell}-n$-divide over $B^{i_{0}}$.\end{proof}

\section{Appendix - various claims.}
\begin{claim}
\label{cla:find_in_ds}Let $\kappa$ be a cardinal, $f:{\rm ds}(\kappa^{+})\to\kappa$.
We can find a sequence\\ $\left\langle \alpha_{k}:k<\omega\right\rangle \subseteq\kappa$
such that for every $k_{\ast}<\omega$ there exists $\eta\in{\rm ds}(\kappa^{+})$
of length $k_{\ast}$ such that $f(\eta\upharpoonright k)=\alpha_{k}$
holds for all $k<k_{\ast}$.\end{claim}
\begin{cor}
\label{cor:extend_indisc_seq}If $M$ is $\kappa$-homogeneous and
$\kappa$-saturated, and $I^{\prime}\supseteq I$ are linear orders
such that $\left|I^{\prime}\right|<\kappa$, $A\subseteq M,\left|A\right|<\kappa$
then: \end{cor}
\begin{enumerate}
\item Every $\left\langle \overline{a}_{t}:t\in I\right\rangle \in{\rm Ind}(\mathfrak{k}^{{\rm or}},A,M)$
can be extended to\\ $\left\langle \overline{a}_{t}:t\in I^{\prime}\right\rangle \in{\rm Ind}(\mathfrak{k}^{{\rm or}},A,M)$.
\item Every $\left\langle \overline{a}_{t}:t\in I\cup\!^{<n}I\right\rangle \in{\rm Ind}(\mathfrak{k}^{{\rm or}+{\rm or}(<n)},A,M)$
can be extended to\\ $\left\langle \overline{a}_{t}:t\in I^{\prime}\cup\!^{<n}I^{\prime}\right\rangle \in{\rm Ind}(\mathfrak{k}^{{\rm or}},A,M)$..
\item Every $\left\langle \overline{a}_{t}:t\in\!^{\leq n}I\right\rangle \in{\rm Ind}(\mathfrak{k}^{{\rm or}(\leq n)},A,M)$
can be extended to\\ $\left\langle \overline{a}_{t}:t\in\!^{\leq n}I^{\prime}\right\rangle \in{\rm Ind}(\mathfrak{k}^{{\rm or}(\leq n)},A,M)$.
\item Every structure $\left\langle \overline{a}_{t}:t\in\mathtt{A}(I)\right\rangle $
indiscernible over $A$ can be extended to\\ $\left\langle \overline{a}_{t}:t\in\mathtt{A}(I^{\prime})\right\rangle $,
also indiscernible over $A$.
\end{enumerate}
\bibliographystyle{alpha}
\addcontentsline{toc}{section}{\refname}\bibliography{master}

\end{document}